\documentclass[12pt]{article}
\usepackage{epsfig,psfrag,amsmath,amssymb,latexsym}
\usepackage{amscd }
\usepackage{amsfonts}
\usepackage{graphicx}
\usepackage{mathrsfs}
\usepackage{xcolor}
\newtheorem{theorem}{Theorem}[section]

\newtheorem{example}{Example}[section]

\newtheorem{lemma}{Lemma}[section]

\newenvironment{proof}[1][Proof]{\textbf{#1.} }{\ \rule{0.5em}{0.5em}}

\DeclareMathOperator{\T}{T}
\DeclareMathOperator{\Cov}{Cov}
\DeclareMathOperator{\Diag}{Diag}
\DeclareMathOperator{\I}{I}

\DeclareMathOperator{\Prob}{P}

\DeclareMathOperator{\Bigo}{O}
\DeclareMathOperator{\Normal}{N}

\DeclareMathOperator{\spectralGap}{gap}
\DeclareMathOperator{\e}{e}
\DeclareMathOperator{\sign}{sign}

\newcommand{\N}{\mathbb{N}}

\newcommand{\Covtrue}{\Sigma}
\numberwithin{equation}{section}
\begin{document}
\title{Quantifying the Estimation Error of Principal Components}
\author{Raphael Hauser
\thanks{Mathematical Institute, Radcliffe Observatory Quarter, Woodstock Road, Oxford OX2 6GG, U.K.;
Alan Turing Institute, British Library, 96 Euston Road, London NW1
2DB, U.K.; Pembroke College, St Aldates, Oxford OX1 1DW, U.K.;
hauser@maths.ox.ac.uk}, J\"uri Lember
\thanks{Institute of Mathematics and Statistics, University of Tartu, Tartu, Estonia
},
Heinrich Matzinger
\thanks{School of Mathematics, Georgia Institute of Technology, Atlanta, GA 30332, USA.
matzi@math.gatech.edu},
Raul Kangro
\thanks{Institute of Mathematics and Statistics, University of Tartu, Tartu, Estonia
}
}

\maketitle
\begin{abstract}{\small Principal component analysis is an important pattern recognition and dimensionality reduction
tool in many applications. Principal components are computed as
eigenvectors of a maximum likelihood covariance $\widehat{\Covtrue}$
that approximates a population covariance $\Covtrue$, and these
eigenvectors are often used to extract structural information about
the variables (or attributes) of the studied population. Since PCA
is based on the eigendecomposition of the proxy covariance
$\widehat{\Covtrue}$ rather than the ground-truth $\Covtrue$, it is
important to understand the approximation error in each individual
eigenvector as a function of the number of available samples. The
combination of recent results in \cite{karimvlad2} and
\cite{yu-wang-samworth} yields such bounds. In the present paper we
sharpen these bounds and show that eigenvectors can often be
reconstructed to a required accuracy from a sample of strictly
smaller size order.}
\end{abstract}

\section{Introduction}

Consider a random row vector $\vec{X}=[X_1,X_2,\ldots,X_p]$, defined over a probability space
$(\Omega,\Prob)$ and representing a data sample of $p$
different items of interest, such as the returns of $p$ different financial assets over a given investment
period, the relative frequencies of $p$ different words in a randomly chosen text, or the expression rates of
$p$ different genes in a cell line exposed to a randomly chosen chemical compound. In many
applications the data are approximately Gaussian with some unknown ground-truth covariance
matrix $\Covtrue=\Cov(\vec{X})$.
The subspaces spanned by the eigenvectors corresponding to the $k$ largest eigenvalues of
$\Covtrue$ serve as best representation of the data in a $k$-dimensional space and represent an
important dimensionality reduction technique. Moreover, the eigenvectors themselves are used to
reveal underlying structure hidden in the data, such as subsets of genes that tend to be jointly expressed, clusters
of texts that belong to a same category, risk factors that drive a financial market and many other
quantities of interest in numerous applications. Often, relevant hidden structure is revealed by the
eigenvectors corresponding to the second to $k$-th largest eigenvalues of $\Covtrue$, but not by the leading
eigenvector. For example, the relative size of components in the leading eigenvector of the covariance
matrix of monthly returns of S\&P500 stocks are approximating the relative weights of the market
capitalizations of these companies, hence this vector reproduces the market index. The components of the
next few leading eigenvectors reveal nearly market neutral investment portfolios, that is, portfolios that are
relatively unaffected by market shocks. Other examples occur in meteorological data, where the second to $k$-th
leading principal components are known to be useful in the automatic detection of nascent storm systems,
and in the natural language processing context, where these vectors can be used to cluster texts by
topic. In genome expression data, all eigenvalues of $\Covtrue$ and associated eigenvectors (respectively, the
singular spectrum and associated singular vectors) can be metabolistically relevant, see e.g.\ \cite{alter2, alter3, alter4}.
For this reason, we are interested in approximating the first few leading eigenvectors of $\Covtrue$,
not just the first.

In practice, a population covariance $\Covtrue$ is of course rarely available and in all likelihood only
exists in the mind of the human modeller. True data are not exactly Gaussian either,  nor can they be expected
to have covariances that are invariant over time. In these circumstances it is customary to collect sample
data vectors $\vec{X}^{(1)},\dots,\vec{X}^{(n)}$, assumed to be i.i.d.\ and taken from the same underlying
distribution. Upon centralizing\footnote{The data are centralized by subtracting their mean and
considering the new data vectors $\vec{Z}^{(i)}=\vec{X}^{(i)}-\frac{1}{n}\sum_{i}\vec{X}^{(i)}$, so that the
expectation of $\vec{Z}^{(i)}$ may be assumed to be negligible for $n$ large enough. For applications in which
the signal to noise ratio is low, such as in daily returns of financial assets, this centralisation step is not strictly
required.}, the data vectors may be assumed to have zero mean. A spectral decomposition of the maximum
likelihood covariance\footnote{For notational simplicity we will use the maximum likelihood covariance throughout
this paper instead of the unbiased covariance estimator, since the two matrices only differ by the factor
$n/(n-1)$ and we assume $n$ to be typically of order $\Bigo(10^2)$ or larger.}
\begin{equation*}
\widehat{\Covtrue}:=
\frac{1}{n}\sum_{i=1}^{n}\left[\vec{X}^{(i)}\right]^{\T}\vec{X}^{(i)}
\end{equation*}
is then used as a proxy for $\Covtrue$, and the leading part of the spectral decomposition of $\widehat{\Covtrue}$
is computed. This approach is called {\em principal component analysis} (PCA). In a situation
where $\Covtrue$ changes over time, such as in finance, one can only avail a limited number of data
points that can meaningfully be considered to have been sampled from the same underlying distribution, and since
the PCA is computed from the maximum likelihood covariance $\widehat{\Covtrue}$, it is important to understand
to what accuracy the computed principal components approximate the eigenvectors of the ground-truth
$\Covtrue$. Thus, we need to study how many sample points
suffice to approximate the eigenvectors of $\Covtrue$ by the eigenvectors of $\widehat{\Covtrue}$ to a given
accuracy. This is an interesting question, as we shall see that not all eigenvectors require the same
number of samples, and that the required sample size depends on the distribution of eigenvalues of the
population covariance $\Covtrue$.

We will henceforth assume that $\vec{X}^{(i)}$, $(i=1,\dots,n)$ are i.i.d.\ copies of
$\Normal(0,\Covtrue)$ random row vectors. Combining these vectors into a data matrix $X$ by row stacking, i.e.,
$\vec{X}^{(i)}$ is taken as the $i$-th row of $X$, the maximum likelihood covariance is given by
$\widehat{\Covtrue}=\frac{1}{n}X^{\T} X$
and determines the estimation error $E:=\widehat{\Covtrue}-\Covtrue$. Let $\Covtrue =
Q\Diag\left(\lambda_1,\dots,\lambda_p\right)Q^{\T}$
be a spectral decomposition of the population covariance in which the eigenvalues appear in
non-increasing
order  $\lambda_1\geq\dots\geq\lambda_p\geq 0$. For ease of exposition, we will now change coordinates and
express all vectors with respect to the basis given by the columns of the orthogonal factor $Q$, that is, the
eigenvectors of $\Covtrue$. Instead of working with the data vectors $\vec{X}^{(i)}$, we thus work with
$\vec{Y}^{(i)}=\vec{X}^{(i)}Q$ and may assume that $\Covtrue=\Diag(\lambda_1,\dots,\lambda_p)$ and
\begin{equation*}
\widehat{\Covtrue}=\frac{1}{n}\sum_{i=1}^n\left[\vec{Y}^{(i)}\right]^{\T}\vec{Y}^{(i)}.
\end{equation*}
This transformation entails neither a change of the operator norm of the estimation error $E$, nor of the
eigenvalues of $\Covtrue$
or $\widehat{\Covtrue}$, nor of any inner products between vectors, as $Q$ is orthogonal, but the new coordinate
system renders the analysis notationally
simpler, since all off-diagonal coefficients of $\Covtrue$ are zero and its eigenvectors
are the canonical unit vectors $\vec{\mu}_i$, $(i=1,\dots,p)$.

In a recent paper that builds on deep results from the theory of random matrices,  Koltchinskii \& Lounici proved
the existence of a universal constant $C_1$, independent of $p$ and $n$, such that the operator norm of
$E=\widehat{\Covtrue}-\Covtrue$ is probabilistically bounded as follows,
\begin{equation}\label{prenana}
\Prob\left[\|E\|\leq C_1\|\Covtrue\|\max\left(\sqrt{\dfrac{r(\Covtrue)}{n}},\dfrac{r(\Covtrue)}{n},\sqrt{\dfrac{t}{n}},\dfrac{t}{n}\right)\right]\geq1-\e^{-t},\ t\geq 1
\end{equation}
where
\begin{equation*}
r(\Covtrue)=\dfrac{\sum_{j=1}^p\lambda_j}{\max_{j=1,\dots,p}\lambda_j}=\dfrac{\sum_{j=1}^p\lambda_j}{\lambda_1}
\end{equation*}
is the {\em effective rank} of $\Covtrue$; see Corollary 2 in \cite{karimvlad2}. 
Note that $r(\Covtrue)\leq p$, so that for $n\geq \max{(p,t)}$ \eqref{prenana} becomes
\begin{equation*}
\Prob\left[\|E\|\leq C_1\|\Covtrue\|\max\left(\sqrt{\dfrac{r(\Covtrue)}{n}},\sqrt{\dfrac{t}{n}}\right)\right]\geq1-\e^{-t},
\end{equation*}
and using $\|\Covtrue\|=\lambda_1$ and the definition of $r(\Covtrue)$, this yields the bound
\begin{equation}\label{nanana}
\Prob\left[\|E\|\leq C_1\frac{\sqrt{\lambda_1}}{\sqrt{n}}\times\max\left(\sqrt{\lambda_1t},
\sqrt{\sum_{j=1}^p\lambda_j}\right)\right]\geq 1-\e^{-t}.
\end{equation}
Previous bounds on $\|E\|$ had only been known for the case where
all eigenvalues are of the same order or when $p$ remains bounded as
$n$ tends to infinity, the so-called {\em finite-dimensional case},
while the result  in  \cite{karimvlad2} applies to the {\em
infinite-dimensional case}, where $p$ tends to infinity as well.

In this paper we study the error that is incurred in approximating individual eigenvectors
of the population covariance $\Covtrue$ by computing a spectral decomposition of the maximum likelihood
covariance matrix $\widehat{\Covtrue}$. Let $\widehat{\lambda}_i$ be the $i$-th largest eigenvalue of
$\widehat{\Covtrue}$ and $\vec{\eta}_i$ an associated eigenvector. For convenience, we assume throughout
that the corresponding eigenvalue $\lambda_{i}$ of the population covariance is non-repeated.
It follows from Corollary 1 of \cite{yu-wang-samworth} 
that the principal angle $\theta(\omega)=\angle(\vec{\mu}_i,\vec{\eta}_i(\omega))$ is bounded by
\begin{equation}\label{yu's bound}
\sin\theta(\omega)\leq\dfrac{2\|E(\omega)\|}{\spectralGap_i}\quad(\omega\in\Omega),
\end{equation}
where $\|E(\omega)\|$ is the operator norm of $E=\Covtrue-\widehat{\Covtrue}(\omega)$ and
$\spectralGap_i:=\min_{j\neq i}|\lambda_i-\lambda_j|$ is the spectral gap at eigenvalue $\lambda_i$ of $\Covtrue$.
Note that in contrast to the classical Davis-Kahan $\sin\theta$-Theorem \cite{davis-kahan} 
the denominator of \eqref{yu's bound} does not depend on the (stochastic) spectrum of $\widehat{\Covtrue}$ but only on
the (deterministic) spectrum of $\Covtrue$.

A combination of
\eqref{yu's bound} with \eqref{nanana} shows that for a fixed $q\in(0,1),\,t\geq 1$, $p\in\N$ and
\begin{equation}\label{caffebean}
n\geq\max\left(p,\;t,\;\dfrac{4C_1^2\,\lambda_1^2\,\,t}{q^2\,\spectralGap^2_i},\dfrac{4C_1^2\,\lambda_1\,\,\sum_{j=1}^p\lambda_j}{q^2\,\spectralGap^2_i}\right),
\end{equation}
we have
\begin{equation*}
\Prob\left[\sin\theta< q\right]\geq 1-\e^{-t}.
\end{equation*}
The aim of the present paper is to show that the bound \eqref{caffebean} on the required sample size
can be replaced by
\begin{equation}\label{improved bound}
n\geq\max\left(p, t, M\times 16C_1^2\max\left((1+\frac{\lambda_i}{\spectralGap_{i}} ),\frac{2\lambda_i}{ q^2\spectralGap_{i}}\right)\,
\max\left((1+\frac{\lambda_i}{\spectralGap_{i}})\,t,\frac{3}{2}\sum_{j\not=i}\frac{\lambda_j}{|\lambda_j-\lambda_{i}|}\right)\right),
\end{equation}
where $M>0$ is a constant that depends on the distribution of the
spectrum $(\lambda_j)$.  For some distributions of the
eigenvalues of $\Covtrue$, this bound is a strict improvement over
the previously known ones. The following is our main result:
 \begin{theorem}\label{main}
For fixed $q\in(0,1)$, $t\geq 1$, $p\in\N$, $i\in \{1,\ldots,p\}$ and $n\geq \max(t,p,\Psi(q,t,i))$ we have
\begin{equation}\label{kaffeebohne}
\Prob\left[\sin\theta\leq q\;\Bigr|\;n\geq \Psi\left(q,t,i^*\right)\right]
\geq 1-\dfrac{p\e^{-t}}{\Prob\left[n\geq\Psi(q,t,i^*)\right]},
\end{equation}
where $i^*\in\arg\min_{k\in\{1,\dots,p\}}|\lambda_k - \widehat{\lambda}_{i}|$ and
\begin{equation*}
\Psi(q,t,i):=16C_1^2\max\left((1+\frac{\lambda_i}{\spectralGap_{i}} ),
2\frac{\lambda_i}{q^2\spectralGap_{i}}\right)\,\max\left((1+\frac{\lambda_i}{\spectralGap_{i}})\,t,\frac{3}{2}\sum_{j\not=i}\frac{\lambda_j}{|\lambda_j-\lambda_{i}|}\right).
\end{equation*}
\end{theorem}

Note that for fixed $p$ and nonrepeated $\lambda_i$, we have $i^*=i$ for $n$ large enough with high probability.
In all relevant contexts to which Theorem \ref{main} applies, one can usually quantify
\begin{equation*}
\Prob\left[n\geq\Psi(q,t,i^*)\right]\approx 1
\end{equation*}
as being close to $1$ for $n\geq M\times\Psi(q,t,i)$ quite easily, where the constant $M>0$ depends on the distribution of
the spectrum $(\lambda_j)$. Substituting such a quantitative bound into \eqref{kaffeebohne}, this yields a lower bound on
\begin{align*}
\Prob\left[\sin\theta\leq q\right]&\geq
\Prob\left[\sin\theta\leq q\;\Bigr|\;n\geq\Psi(q,t,i^*)\right]\times
\Prob\left[n\geq\Psi(q,t,i^*)\right]\\
&\stackrel{\eqref{kaffeebohne}}{\geq}\Prob\left[n\geq\Psi(q,t,i^*)\right]-p\e^{-t}
\end{align*}
for all $n$ that satisfy the bound \eqref{improved bound}.
Since a comprehensive treatment of all cases is not a reasonable undertaking, we limit the scope of this paper to the
derivation of the bound \eqref{kaffeebohne} and leave the quantification of the constant $M$ to users of Theorem \ref{main}.

For the purposes of our analysis we will choose the eigenvector $\vec{\eta}_i$ such that
$\langle \vec{\mu}_i,\vec{\eta}_i\rangle> 0$ and $\vec{\eta}_i\in\vec{\mu}_i + (\vec{\mu}_i)^{\perp}$, a choice that is
possible with probability $1$, and we shall write
$\Delta \vec{\mu}_i = \vec{\eta}_i-\vec{\mu}_i$, so that
\begin{equation}\label{halloechen}
\sin\theta(\omega)\leq\tan\theta(\omega)=\|\Delta\vec{\mu}_i\|_2.
\end{equation}
Therefore, it suffices to prove that
\begin{equation}\label{scopeStar}
\Prob\left[\|\Delta\vec{\mu}_i\|_2<q\,\Big|\, n\geq\max\left(pt,\,\Psi\left(q,t,i^*\right)\right)\right]\geq
1-\dfrac{p\e^{-t}}{\Prob\left[n\geq\Psi(q,t,i^*)\right]}
\end{equation}
to establish Theorem \ref{main}.

For eigenvalues $\lambda_i$ of lower order, the new bound \eqref{improved bound} is of strictly smaller order than
\eqref{caffebean}, because the numerator satisfies $\lambda_{i}\ll\lambda_1$, and furthermore,
the eigenvalues of largest order in the sum $\sum_{j=1}^p \lambda_j$ are replaced by terms of $\Bigo(1)$ in the sum
$\sum_{j\neq i}\lambda_j/|\lambda_j-\lambda_{i}|$, so that this sum is generally much smaller when the
eigenvalue $\lambda_{i}$ is of medium (but not smallest) order. Note that, like the $\sin\theta$-Theorem of
Yu-Wang-Samworth \cite{yu-wang-samworth}, our bound \eqref{improved bound} only depends on the spectrum of
the population covariance, but not on the maximum likelihood covariance.

\begin{example}
In data science one often deals with data matrices whose scree plots have a short and quickly decreasing
initial section, followed by a large slowly decreasing middle section with close to equally spaced eigenvalues,
and if the population covariance is a noisy version of a rank $k<p$ matrix, then this is followed by a sharp drop
off to background noise after position $k$. For an example, see e.g. \cite{alter1} or \cite{webpage} 

A reasonable model is therefore to assume that for some fixed $\beta\in(0,1)$, $\Covtrue$ has $\Bigo(1)$
eigenvalues of order $\Bigo(p)$ that are structurally not important, $\Bigo(p^{1-\beta})$ eigenvalues of order
$\Bigo(p^\beta)$, and $\Bigo(p)$ smaller eigenvalues of order $\Bigo(1)$,
where $\Covtrue$ has been rescaled so that $\sum_{j=1}^p \lambda_j =p$. The latter assumption is reasonable,
as PCA is often carried out on the correlation matrix rather
than the covariances, to correct for differences in scale between the feature variables.

The relevant structural information is often revealed by eigenvectors corresponding to eigenvalues of the intermediate
order $\Bigo(p^{\beta})$ rather than those of the largest order $\Bigo(p)$. For example, when PCA is used to classify
texts into different topics, the principal components that correspond to the largest order of eigenvalues typically pick up
high frequency words that are common in all texts of the same language, whereas the topic is typically identified by
context-specific words that merely appear with higher frequency than in the general text population.

Assuming that the eigenvalues of order $\Bigo(p^\beta)$ are locally not very different
from a renewal process, or similarly, that they are approximately located on a lattice, their spectral gaps
are typically of order $\Bigo\bigl(p^\beta/p^{1-\beta}\bigr)=\Bigo(p^{2\beta-1})$.
In the case where $\lambda_i$ is of order $\Bigo(p^{\beta})$, this yields
\begin{align}
\sum_{\lambda_j=\Bigo(p)}\dfrac{\lambda_j}{\left|\lambda_j-\lambda_i\right|}&=\Bigo\left(1\right),\nonumber\\
\sum_{\lambda_j=\Bigo\left(p^\beta\right),\,j\not=i}\dfrac{\lambda_j}{\left|\lambda_j-\lambda_i\right|}
&=\Bigo\left(\frac{p^\beta}{p^{2\beta-1}}\times \log p\right)=\Bigo\left(\log p \times p^{1-\beta}\right),\label{hansjoggeli}\\
\sum_{\lambda_j=\Bigo(1)}\dfrac{\lambda_j}{|\lambda_j-\lambda_i|}
&=\Bigo\left(p\times p^{-\beta}\right)=\Bigo\left(p^{1-\beta}\right). \nonumber
\end{align}
where the left-hand side of \eqref{hansjoggeli} is a harmonic series under the assumption that the eigenvalues
of order $\Bigo(p^{\beta})$ lie on a lattice. Summing all three bounds, we obtain
\begin{equation*}
\sum_{j=1,\,j\not=i}^p \dfrac{\lambda_j}{\left| \lambda_j -\lambda_i\right|}=\Bigo\left(p^{1-\beta}\times\log p\right).
\end{equation*}

Substituting this bound into \eqref{improved bound}, we find that
the number of samples required to guarantee that $\sin\theta<q$ with
probability larger than $1-\varepsilon$ when $\lambda_{i}$
corresponds to one of the eigenvalues of intermediate order
$\Bigo(p^{\beta})$ is for $\beta>\frac{1}{2}$ given by
\begin{align*}
n&=\Bigo\left(\max(-\log \varepsilon+\log p,p, p^{1-\beta}\times p^{1-\beta}\times \log p)\right)=\Bigo\left(\max(\log p,p,p^{2(1-\beta)}\log p) \right)=\Bigo(p),
\end{align*}
 while the bound \eqref{caffebean} would imply that the number of
samples required is of the order
\begin{equation*}
n=\Bigo\left(\max\big(-\log p ,p, \dfrac{\log p\times p^2 }{p^{4\beta-2}}\times\dfrac{p^2}{p^{4\beta-2}}\big)\right)=
\Bigo\left( p^{4(1-\beta)}\log p\right)
\end{equation*}
which is of strictly larger order for $\beta<\frac{3}{4}$.
Thus, for $\frac{1}{2}<\beta<\frac{3}{4}$ we can reconstruct the
relevant principal components with linearly many samples, whereas
the bound \eqref{caffebean} would have suggested that the number of
samples required is superlinear in the number $p$ of feature
variables.
\end{example}

\section{Outline of the Proof}

Before we proceed with a detailed proof of Theorem \ref{main}, it is
helpful to outline of the main ideas. In what follows,  we shall
assume that $\lambda_i$ is unrepeated, and we will write
$\Delta\lambda_{k,i}=\widehat{\lambda}_i-\lambda_k$, where $k$ is a
fixed index that we shall specify later. In analogy to the earlier
construction of the vector $\Delta\vec{\mu}_i$, let $\vec{\eta}_i$
be an eigenvector of $\widehat{\Covtrue}$ associated with the $i$-th
largest eigenvalue $\widehat{\lambda}_i$, and chosen such that
$\langle\vec{\mu}_k,\vec{\eta}_i\rangle\geq0$ and
$\vec{\eta}_i\in\vec{\mu}_k+\vec{\mu}_k^{\perp}$, and let us write
$\Delta\vec{\mu}_{k,i}:=\vec{\eta}_i-\vec{\mu}_k$, so that
$\Delta\vec{\mu}_{k,i}\in\vec{\mu}_k^{\perp}$. In the special case
where $k=i$ we denote $\Delta\vec{\mu}_{i}=\Delta\vec{\mu}_{i,i}$
and  $\Delta\lambda_i:=\Delta\lambda_{i,i}$.

For fixed $p$ and $k=i$,
\begin{equation}\label{approxerroreigen}
\Delta\vec{\mu}_i\approx
-\sqrt{\dfrac{\lambda_i}{n}}\; Z_i,
\end{equation}
where $Z_i=[Z_{1,i},\dots,Z_{p,i}]$ is a random vector of size $p$ with coefficients
\begin{equation*}
Z_{j,i}=\begin{cases}\dfrac{\sqrt{\lambda_j}}{\lambda_j-\lambda_i}\;N_{j,i},&\quad(j\neq i),\\
0,&(j=i),
\end{cases}
\end{equation*}
and where the random variables $N_{j,i}$ converge (jointly, in
distribution) to independent standard Gaussians as
$n\rightarrow\infty$ (see, e.g. \cite{anderson}, Thm. 13.5.1). Thus,
to guarantee that $\|\Delta\vec{\mu}_i\|<q$ with high probability,
we need $\sqrt{\lambda_i/n}\times\|Z_i\|<q$ with high probability.
Assuming $n$ to be large enough for the variables $N_{i,j}$ to be
close to i.i.d.\ standard Gaussians and the approximation
\eqref{approxerroreigen} to hold, one finds
\begin{equation*}
\|\Delta\vec{\mu}_i\|^2\approx\dfrac{1}{n}\times\sum_{j\neq i}
\frac{\lambda_i\lambda_j}{(\lambda_j-\lambda_i)^2}N_{j,i}^2
\end{equation*}
and this approximation can be used to bound $P(\|\Delta\mu_i\|<q)$.
In the present case, when $p$ is not fixed, this approach will not
holds, but let us gain a quick understanding of how
\eqref{approxerroreigen} arises in the finite-dimensional case in
which $n$ is left to tend to infinity for fixed $p$, and of how the
argument has to be amended in the infinite-dimensional case, where
$p$ is allowed to tend to $\infty$ at a linear rate in $n$: We have
\begin{align}
\Covtrue\;\vec{\mu}_k&=\lambda_k\vec{\mu}_k,\label{cocoisteincacaesser}\\  
\widehat{\Covtrue}\;[\vec{\mu}_k+\Delta\vec{\mu}_{k,i}]&=(\lambda_k+\Delta\lambda_{k,i})
[\vec{\mu_k}+\Delta\vec{\mu}_{k,i}].\label{thehiddenmessageforraphael} 
\end{align}
Subtracting \eqref{cocoisteincacaesser} from \eqref{thehiddenmessageforraphael} and using
$\widehat{\Covtrue}=\Covtrue+E$ yields
\begin{equation}\label{coco}
[\Covtrue-(\lambda_k+\Delta\lambda_{k,i})\I_p]\;\Delta\vec{\mu}_{k,i}+
E\Delta\vec{\mu}_{k,i}=-E\vec{\mu}_k
+\Delta\lambda_{k,i}\vec{\mu}_k
\end{equation}
where $\I_p$ is the $p\times p$ identity matrix. In the finite-dimensional case, where $p$ is fixed and $n$
tends to infinity, taking $k=i$ implies that $E$, $\Delta\vec{\mu}_{i}$ and $\Delta\lambda_{i}$ are all of order
$1/\sqrt{n}$, and hence the terms $\Delta\lambda_i\Delta\vec{\mu}_i$ and $E\Delta\vec{\mu}_i$ are of the smaller
order $1/n$ and can be neglected in the asymptotics, so as to give the approximation
\begin{equation}\label{coco2}
[\Covtrue-\lambda_i\I_p]\;\Delta\vec{\mu}_i
\approx-E\vec{\mu}_i
+\Delta\lambda_i \vec{\mu}_i
\end{equation}
Using the facts that $\Covtrue=\Diag(\lambda_j)$, $\vec{\mu}_i$ is the $i$-th unit vector and that $\vec{\mu}_i$ and
$\Delta\vec{\mu}_i$ are mutually orthogonal by construction, the $i$-th equation of system \eqref{coco2} yields
\begin{equation*}
\Delta\lambda_i\approx E_{ii},
\end{equation*}
and dividing the $j$-th equation of the system \eqref{coco2} by $\lambda_j-\lambda_i$ $(j\neq i)$ yields
\begin{equation}\label{FINITE}
\Delta\vec{\mu}_i\approx -\begin{pmatrix}
\dfrac{E_{1,i}}{\lambda_1-\lambda_i}\\
\dfrac{E_{2,i}}{\lambda_2-\lambda_i}\\
\vdots\\
\dfrac{E_{(i-1),i}}{\lambda_{i-1}-\lambda_i}\\
0\\
\dfrac{E_{(i+1),i}}{\lambda_{i+1}-\lambda_i}\\
\vdots\\
\dfrac{E_{p,i}}{\lambda_p-\lambda_i}
\end{pmatrix}
=-\sqrt{\dfrac{\lambda_i}{n}}\begin{pmatrix}
\dfrac{\sqrt{\lambda_1}}{\lambda_1-\lambda_i}N_{1,i}\\
\dfrac{\sqrt{\lambda_2}}{\lambda_2-\lambda_i}N_{2,i}\\
\vdots\\
\dfrac{\sqrt{\lambda_{i-1}}}{\lambda_{i-1}-\lambda_i}N_{(i-1),i}\\
0\\
\dfrac{\sqrt{\lambda_{i+1}}}{\lambda_{i+1}-\lambda_i}N_{(i+1),i}\\
\vdots\\
\dfrac{\sqrt{\lambda_p}}{\lambda_p-\lambda_i}N_{p,i}
\end{pmatrix},
\end{equation}
where
\begin{equation}\label{define N}
N_{s,t}:=\sqrt{\dfrac{n}{\lambda_s \lambda_t}}E_{s,t}
\end{equation}
for all $s,t\in 1,\ldots,p$ with $s\neq t$. The random variables $N_{ji}$ $(j\neq i)$
converge in joint distribution to i.i.d.\ standard Gaussians, a fact we only mention for motivational purposes and on
which our technical argument below does not rely.

In contrast, in the infinite-dimensional case the terms $\Delta\lambda_{k,i}\Delta\vec{\mu}_{k,i}$ and
$E\Delta\vec{\mu}_{k,i}$ can no longer be asymptotically disregarded, as $p$ is also allowed to tend to infinity
at up to a linear rate in $n$.
Let $P_k$ denote the orthogonal projection into the orthogonal complement $\vec{\mu}_k^{\perp}$ of
$\vec{\mu}_k$, and define the operator
\begin{align*}
D_k:=-{\rm diag}\Big({1\over \lambda_1-\hat{\lambda}_i},\ldots,{1\over
    \lambda_{k-1}-\hat{\lambda}_i},0,{1\over
    \lambda_{k+1}-\hat{\lambda}_i},\ldots,{1\over \lambda_p-\hat{\lambda}_i}\Big),
\end{align*}
so that $D_k\,\vec{\mu}_k=0$ and $D_k$ is well defined
as long as
$\hat{\lambda}_i$ is not an eigenvalue (except, possibly, the $k-$th eigenvalue)
of $\Covtrue$. Using the fact that $\Delta\vec{\mu}_{k,i}\in\vec{\mu}_k^{\perp}$, multiplying \eqref{coco} by
$D_k$ and solving for $\Delta\mu_{k,i}$ yields
\begin{align}
\Delta\vec{\mu}_{k,i}&=-(\I_p -D_k\,E\,P_k)^{-1}\,D_k
\begin{pmatrix}
E_{1k}&\ldots&E_{(k-1)k}&0&E_{(k+1)k}&\ldots&E_{pk}
\end{pmatrix}^{\T}\nonumber\\
&=-\sqrt{\dfrac{\lambda_k}{n}}\times
(\I_p - D_k\,E\,P_k)^{-1}\,\begin{pmatrix}
\frac{\sqrt{\lambda_1}}{\lambda_1-\hat{\lambda}_i}N_{1,k}\\
\frac{\sqrt{\lambda_2}}{\lambda_2-\hat{\lambda}_i}N_{2,k}\\
\ldots\\
\frac{\sqrt{\lambda_{k-1}}}{\lambda_{k-1}-\hat{\lambda}_i}N_{(k-1),k}\\
0\\
\frac{\sqrt{\lambda_{k+1}}}{\lambda_{k+1}-\hat{\lambda}_i}N_{(k+1),k}\\
\ldots\\
\frac{\sqrt{\lambda_p}}{\lambda_p-\hat{\lambda}_i}N_{p,k}
\end{pmatrix}\label{leberwurst}\\
&=-\sqrt{\dfrac{\lambda_k}{n}}\times
(\I_p - D_k\,E\,P_k)^{-1}\,|D_k|^{1/2}J\vec{V}_k,\label{leberwurstStar}
\end{align}
where $N_{s,t}$ is as defined
in \eqref{define N}, $\vec{V}_k=(v_{j,k})$ is defined by
\begin{equation}\label{definition of v}
v_{j,k}:=\sqrt{\frac{\lambda_{j}}{\left|\lambda_j-\hat{\lambda}_i\right|}}\times N_{j,k},\quad(j=1,\dots,k-1,k+1,\dots,p),\ v_{k,k}:=0,
\end{equation}
$|D_{k}|$ is the matrix obtained by replacing the coefficients
of $D_{k}$ by their absolute values, and $J$ the diagonal
matrix with coefficients $J_{j,j}:=\sign (D_k)_{j,j}$.

\begin{example}\label{finiteD}
Let $\vec{X}=(X_1,X_2,X_3)$ be a $3$-dimensional random vector with
zero mean and covariance matrix
$\Covtrue=\Diag(\lambda_1,\lambda_2,\lambda_3)$, where
$\lambda_1>\lambda_2>\lambda_3$. Using the notation introduced
earlier, and taking $n$ large enough to guarantee that $1^*=1$ with
high probability, the first line of \eqref{coco} for $k=i=1$
yields with
$\vec{\eta}_1=(1,\eta_{21},\eta_{31})$
\begin{equation*}
\hat{\lambda}_1-\lambda_1=\Delta\lambda_{1}=E_{11}+E_{12}\eta_{21}+E_{13}\eta_{31}=
\dfrac{\lambda_1}{\sqrt{n}} N_{11} +
\frac{\sqrt{\lambda_1}}{\sqrt{n}}\big(\sqrt{\lambda_2}N_{12}\eta_{21}+
\sqrt{\lambda_3}N_{13}\eta_{31}\big),
\end{equation*}
whereas the second and third lines yield
\begin{equation*}
\left[ \left(
\begin{array}{cc}
\lambda_2-\hat{\lambda}_1 & 0  \\
0 & \lambda_3-\hat{\lambda}_1 \\
\end{array}
\right)
+
\left(
\begin{array}{cc}
E_{22} & E_{23} \\
E_{32} & E_{33} \\
\end{array}
\right) \right]
\,
\left(
\begin{array}{c}
\eta_{21} \\
\eta_{31} \\
\end{array}
\right)
= -\left(
\begin{array}{c}
E_{21} \\
E_{31} \\
\end{array}
\right).
\end{equation*}
Equivalently
\begin{equation*}
\left[ \left(
\begin{array}{cc}
1 & 0  \\
0 & 1 \\
\end{array}
\right)
+
\left(
\begin{array}{cc}
{E_{22}\over \lambda_2-\hat{\lambda}_1} & {E_{23}\over \lambda_2-\hat{\lambda}_1} \\
{E_{32}\over \lambda_3-\hat{\lambda}_1} & {E_{33}\over \lambda_3-\hat{\lambda}_1} \\
\end{array}
\right) \right]
\,
\left(
\begin{array}{c}
\eta_{21} \\
\eta_{31} \\
\end{array}
\right)
= -\left(
\begin{array}{c}
{E_{21}\over \lambda_2-\hat{\lambda}_1} \\
{E_{31}\over \lambda_3-\hat{\lambda}_1} \\
\end{array}
\right).
\end{equation*}
Therefore, with
$$D_1={\rm diag}\big(0,{-1\over \lambda_2-\hat{\lambda}_1},{-1\over
    \lambda_3-\hat{\lambda}_1}\big),\quad P_1={\rm
    diag}\big(0,1,1\big)$$ we obtain
\begin{align*}
\Delta\vec{\mu}_{1,1}&=
\left(\begin{array}{c}
0\\
\eta_{21} \\
\eta_{31} \\
\end{array}\right)\,=\left[ \left(
\begin{array}{ccc}
1 & 0 & 0 \\
0 & 1 & 0 \\
0 & 0 & 1 \\
\end{array}
\right)+
\left(
\begin{array}{ccc}
0 & 0 & 0\\
0 &  {E_{22}\over \lambda_2-\hat{\lambda}_1} & {E_{23}\over \lambda_2-\hat{\lambda}_1} \\
0 &  {E_{32}\over \lambda_3-\hat{\lambda}_1} & {E_{33}\over \lambda_3-\hat{\lambda}_1} \\
\end{array}
\right) \right]^{-1} 
\left(
\begin{array}{c}
0  \\
{-E_{21}\over \lambda_2-\hat{\lambda}_1} \\
{-E_{31}\over \lambda_3-\hat{\lambda}_1} \\
\end{array} \right)\\
& =[I_3-D_1EP_1]^{-1} 
D_1\left(
\begin{array}{c}
0 \\
E_{21} \\
E_{31} \\
\end{array}
\right).\end{align*}

\end{example}

Comparing \eqref{leberwurst} with formula \eqref{FINITE} from the finite-dimensional case, we note the following
differences:
Firstly \eqref{leberwurst} is an exact formula, whilst \eqref{FINITE} is an approximation in the case where $k=i$.
Secondly, the appearance of the term $\Delta\lambda_{k,i}$ in the denominators on the r.h.s.\ of \eqref{leberwurst} is
problematic, unless we can bound $|\lambda_j-\hat{\lambda}_i|$ away from zero. We will fix this choice
by setting $k=i^*$ and conditioning on $i^*$, where the random index $i^*$ is defined in Theorem \ref{main}.
As we will show in Section \ref{technical details}, this yields the following bounds:

\begin{lemma}\label{inflator lemma}
For all $j\neq i^*$, the following hold true,
\begin{align}
\dfrac{1}{\left| \lambda_j-\hat{\lambda}_i\right|}
&\leq\dfrac{2}{\left| \lambda_j-\lambda_{i^*}\right|},\label{inflator}\\
\dfrac{\lambda_j}{\left| \lambda_j-\lambda_{i^*}\right|}&\leq 1+\dfrac{\lambda_{i^*}}{\spectralGap_{i^*}}.
\label{inflationista}
\end{align}
\end{lemma}

Thirdly, and most significantly, the term $D_{i^*} E\,P_{i^*}$
appears in the r.h.s.\ of \eqref{leberwurst}. If we managed to
prove that $\|D_{i^*} E\,P_{i^*}\|  < 1$, then by the
Neuman Series Formula,
\begin{equation}\label{run}
(\I_{p}- D_{i^*} E\,P_{i^*})^{-1}=\I_{p} + D_{i^*}E\,P_{i^*} + (D_{i^*}E\,P_{i^*})^2 + (D_{i^*}E\,P_{i^*})^3 +\dots,
\end{equation}
and then we could argue along the lines of the finite-dimensional case.
However, instead of bounding $\|D_{i^*} E\,P_{i^*}\|$, we will bound $\|\Lambda_{i^*}\|$, defined as
\begin{equation}\label{DED}
\Lambda_{i^*}:=|D_{i^*}|^{1/2} E |D_{i^*}|^{1/2},
\end{equation}
where $|D_{i^*}|$ denotes the matrix obtained by replacing the coefficients of $D_{i^*}$ by their absolute values.
By using the properties
\begin{equation*}
D_{i^*}=|D_{i^*}|^{1/2} J |D_{i^*}|^{1/2},\quad P_{i^*}D_{i^*}=D_{i^*},\quad P_{i^*}=|D_{i^*}|^{1/2} |D_{i^*}|^{-1/2},
\end{equation*}
where $|D_{i^*}|^{-1/2}$ is the  diagonal matrix containing
values $\sqrt{|\lambda_j-\hat{\lambda_i}|}$ at positions
$j\not=i^*$, and 0 at the position $i^*$, we easily get
\begin{equation*}
(D_{i^*}E\,P_{i^*})^k=
|D_{i^*}|^{1/2}\left(J\Lambda_{i^*}\right)^k|D_{i^*}|^{-1/2}, \quad k\geq 1.
\end{equation*}
The Neumann series \eqref{run} may thus be rewritten as
\begin{equation}\label{runStar}
(\I_{p} - D_{i^*} E\,P_{i^*})^{-1}=I_p +
|D_{i^*}|^{1/2}\left(\sum_{k=1}^{\infty}\left(J\Lambda_{i^*}\right)^k\right)|D_{i^*}|^{-1/2}
\end{equation}
implying that
$$(I_p -
D_{i^*} \, E \,P_{i^*})^{-1}\,|D_{i^*}|^{1/2}=|D_{i^*}|^{1/2}\big(I_p+\sum_{k=1}^{\infty}\left(J\Lambda_{i^*}\right)^k\big).$$
If $\|\Lambda_{i^*}\|\in(0,1)$, then this series converges, and using \eqref{inflator} in
\begin{equation*}
\||D_{i^*}|^{1/2}\|=\max_{j\neq i^*}\frac{1}{\sqrt{|\lambda_j-\hat{\lambda}_i}|}
\leq \sqrt{\dfrac{2}{\spectralGap_{i^*}}},
\end{equation*}
the taking of norms on both sides of \eqref{leberwurstStar} yields
\begin{equation}\label{soso}
\|\Delta\vec{\mu}_{i^*,i}\|\leq\sqrt{\dfrac{2}{\spectralGap_{i^*}}}\times
\dfrac{1}{1-\|\Lambda_{i^*}\|}\times\sqrt{\dfrac{\lambda_{i^*}}{n}}\|\vec{V}_{i^*}\|.
\end{equation}

A crucial observation is that for $i^*$ fixed, the norm of $\Lambda_{i^*}$
can be bounded by the norm of a matrix, which has an interpretation as covariance matrix estimation error for a multivariate Gaussian random vector with
zero mean, independent coefficients, and whose $j$-th coefficient has variance
\begin{equation}\label{newmatrix}
\dfrac{\lambda_j}{\left| \lambda_j-\lambda_{i^*}\right|},\quad(j\neq i^*).
\end{equation}
Its operator norm can therefore be bounded using the technique of Koltchinskii \& Lounici that allowed for the derivation
of  \eqref{nanana}. To exploit this mechanism, we will have to condition on the value of $i^*$, but for intuitive purposes,
the reader may keep in mind that  for $n$ that satisfy \eqref{improved bound}, we have $i^*=i$ with high probability.

\section{Technical Details}\label{technical details}

We will now fill in the missing details of the proof we outlined so far. We start with two technical lemmas.

\begin{proof} {\bf (Lemma \ref{inflator lemma})}
By definition of $i^*$, for all $j\neq i^*$ we have

$$| \lambda_j-\lambda_{i^*}|\leq | \lambda_j-\widehat{\lambda}_{i}|+|
\widehat{\lambda}_{i}-\lambda_{i^*}|\leq 2 |
\lambda_j-\widehat{\lambda}_{i}|$$ which shows \eqref{inflator}.
As
\[\frac{\lambda_j}{|\lambda_j -\lambda_{i^*}|}=\frac{\lambda_j-\lambda_{i^*}}{|\lambda_j -\lambda_{i^*}|}+\frac{\lambda_{i^*}}{|\lambda_j -\lambda_{i^*}|}\leq 1+\frac{\lambda_{i^*}}{|\lambda_j -\lambda_{i^*}|},\]
we have also proved the inequality \eqref{inflationista}
\end{proof}

\begin{lemma}\label{rescaling lemma}
Let $\Diag(\nu)$ be a diagonal matrix of size $p$ with diagonal
coefficients $\nu_j\geq 0$, $(j=1,\dots,p)$, and let
$E(\nu)=\Diag(\nu)E\Diag(\nu)$, where
$E=\widehat{\Covtrue}-\Covtrue$ is the covariance matrix error used
throughout this paper. Then for all $t\geq 1$
and $n\in\N,\ n\geq \max(p,t)$ it is true that
\begin{equation*}
\Prob\left[\|E(\nu)\|^2\leq\Xi(n,t,\nu)\right]\geq 1-\e^{-t},
\end{equation*}
where
\begin{equation*}
\Xi(n,t,\nu)=\dfrac{C_1^2 \|\nu^2\lambda\|_\infty}{n}\times\max\left(\|\nu^2\lambda\|_\infty \times t,\sum_{j=1}^p\nu_j^2\lambda_j
\right)
\end{equation*}

with
\[\|\nu^2\lambda\|_\infty=\max_{j=1,2,\ldots,p}\nu_j^2\lambda_j .\]

\end{lemma}

\begin{proof}
Recall that
\begin{equation*}
E = \dfrac{1}{n}\sum_{i=1}^n\left[\vec{Y}^{(i)}\right]^{\T}\vec{Y}^{(i)} - \Covtrue,
\end{equation*}
where $\vec{Y}^{(i)}$ are i.i.d.\ multivariate normal row vectors with zero mean and covariance matrix
$\Covtrue=\Diag(\lambda)$. Therefore,
\begin{equation*}
E(\nu) = \dfrac{1}{n}\sum_{i=1}^n\left[\vec{Z}^{(i)}\right]^{\T}\vec{Z}^{(i)}- \Covtrue(\nu),
\end{equation*}
where $\vec{Z}^{(i)}=\vec{Y}^{(i)}\Diag(\nu)$ are i.i.d.\ multivariate normal random vectors with covariance matrix
$\Covtrue(\nu):=\Diag(\nu)\Covtrue \Diag(\nu)=
\Diag(\nu_1^2\lambda_1,\dots,\nu_p^2\lambda_p)$. Thus, $E(\nu)$ is the
error matrix of the maximum likelihood estimator of $\Covtrue(\nu)$, so that the claim follows from \eqref{nanana}.
\end{proof}

Inequality \eqref{soso} suggests that our aim to prove \eqref{scopeStar} can be achieved by bounding
$\sqrt{\lambda_{i^*}/n}\|\vec{V}_{i^*}\|$ and $\|\Lambda_{i^*}\|$ with high probability. To do this, we need to construct an
appropriate scaling vector $\nu$ and apply Lemma \ref{rescaling lemma}.
Recall that $\vec{V}_{i^*}=(v_{j,i^*})$ satisfies
\begin{align*}
\sqrt{\dfrac{\lambda_{i^*}}{n}}v_{j,i^*}&=\sqrt{\dfrac{\lambda_j\lambda_{i^*}}
{n|\lambda_j-\hat{\lambda_i}}|}\cdot N_{j,i^*}
=\sqrt{\dfrac{1}{|\lambda_j-\hat{\lambda}_i}|}\cdot E_{j,i^*},\quad(j\neq i^*).
\end{align*}
Let $\vec{W}=(w_j)$ be defined by
$w_j:=\tilde{\nu}_{j}\tilde{\nu}_{i^*}E_{j,i^*}$, where
$\tilde{\nu}_{j}:=|\lambda_j-\hat{\lambda}_i|^{-1/2}$ for
$(j\neq i^*)$, and
\[
\tilde{\nu}_{i^*} :=\max_{j\not=i^*} \frac{\sqrt{\lambda_j}}{\sqrt{|\lambda_j-\lambda_{i^*}|}\,\sqrt{\lambda_{i^*}}}.
\]

Under the notation of Lemma \ref{rescaling lemma} the vector $\vec{W}$ is then the $i^*$-th column of $E(\tilde{\nu})$.
By construction, $\tilde{\nu}_{i^*}\cdot\sqrt{\lambda_{i^*}/n}\cdot v_{j,i^*} = w_j$ for $(j\neq i^*)$ and $v_{i^*,i^*}=0$, so that
\begin{equation}
\dfrac{\tilde{\nu}^2_{i^*}}\lambda_{i^*}{n}\|\vec{V}_{i^*}\|^2\leq \|\vec{W}\|^2\leq \|E(\tilde{\nu})\|^2,\label{um}
\end{equation}
Next, recall that
\begin{align*}
\Lambda_{i^*} &= |D_{i^*}|^{1/2}E|D_{i^*}|^{1/2}=P_{i^*}\Diag(\tilde{\nu})E\Diag(\tilde{\nu})P_{i^*}.
\end{align*}
Therefore,
\begin{equation}
\|\Lambda_{i^*}\|^2\leq\|E(\tilde{\nu})\|^2.\label{dois}
\end{equation}
Now let $\nu_{j}:=\sqrt{2}\,|\lambda_j-\lambda_{i^*}|^{-1/2}$
for $(j\neq i^*)$, and $\nu_{i^*}:=\tilde{\nu}_{i^*}$, so that
by virtue of \eqref{inflator}, we have $\nu_j/\tilde{\nu}_j\geq 1$
for all $j$, and
$E(\nu)=\Diag\left((\nu_j/\tilde{\nu}_j)\right)E(\tilde{\nu})\Diag\left((\nu_j/\tilde{\nu}_j)\right)$.
This implies
\begin{equation}\label{treis}
\|E(\tilde{\nu})\|^2\leq\|E(\nu)\|^2.
\end{equation}
 In order to prove our main result, we would like to
show that for $n$ large enough we have $\Delta \mu_{i*,i}\leq q$
with a high probability. Taking into account \eqref{um},
\eqref{dois}, \eqref{treis} and \eqref{soso} this is true, if the
conditions
\begin{equation}\label{Enu_norm1}\|E(\nu)\|\leq \frac{1}{2}\end{equation}
and
\begin{equation}\label{Enu_norm2}\|E(\nu)\|\leq q\cdot {\sqrt{\spectralGap_{i^*}}\cdot\tilde{\nu}_{i^{*}}\over 2\sqrt{2}}\end{equation}
hold with a high probability. Lemma \ref{rescaling lemma} gives us hope that both conditions are satisfied with a high probability, if for a sufficiently large $t$ and for $n\geq \max(p,t)$ we have
\[\Xi(n,t,\nu)\leq \min(\frac{1}{4},{\,q^2 \spectralGap_{i^*}\,\tilde{\nu}_{i^{*}}^2\over 8}).\]
By definition of $\nu$,
$\|\nu^2\lambda\|_{\infty}=2\lambda_{i^*}\tilde{\nu}_{i^*}^2$ and by
the definition of $\tilde{\nu}_{i^*}$,
$$\lambda_{i^*}\tilde{\nu}_{i^*}^2=\max_{j\ne
    i}\frac{\lambda_j}{|\lambda_j-\lambda_{i^*}|}\leq
\sum_{j\not=i^*}\frac{\lambda_j}{|\lambda_j-\lambda_{i^*}|},\quad
\lambda_{i^*}\tilde{\nu}_{i^*}^2\leq \big(1+{\lambda_{i^*}\over
    \spectralGap_{i^*}}\big),$$ where   the last inequality follows from  (\ref{inflationista}). Hence
\begin{align*}
\Xi(n,t,\nu)&=\frac{2C_1^2\lambda_{i^*}\tilde{\nu}_{i^*}^2}{n}\times
\max \Big( 2\lambda_{i^*}\tilde{\nu}_{i^*}^2
t,2\sum_{j\not=i^*}\frac{\lambda_j}{|\lambda_j-\lambda_{i^*}|}+\lambda_{i^*}\tilde{\nu}_{i^*}^2 \Big)\\
&\leq \frac{4C_1^2\lambda_{i^*}\tilde{\nu}_{i^*}^2}{n}\times
\max\Big( \lambda_{i^*}\tilde{\nu}_{i^*}^2
t,\frac{3}{2}\sum_{j\not=i^*}\frac{\lambda_j}{|\lambda_j-\lambda_{i^*}|}\Big)\\
&\leq  \frac{4C^2_1}{n}\times
\big(\lambda_{i^*}\tilde{\nu}_{i^*}^2\big)\times \max\Big(
\big(1+\frac{\lambda_{i^*}}{\spectralGap_{i^*}}\big)
t,\frac{3}{2}\sum_{j\not=i^*}\frac{\lambda_j}{|\lambda_j-\lambda_{i^*}|}
\Big)=:\Phi(n,t,i^*).
\end{align*}
Therefore, we get that both \eqref{Enu_norm1} and
\eqref{Enu_norm2} are satisfied with high probability, if in
addition to $n\geq \max(p,t)$ we have
\[n\geq \Psi(q,t,i^*),\]
where
\[\Psi(q,t,i)=16C_1^2\max\left((1+\frac{\lambda_i}{\spectralGap_{i}} ),
2\frac{\lambda_i}{
    q^2\spectralGap_{i}}\right)\,\max\left((1+\frac{\lambda_i}{\spectralGap_{i}})\,t,{3\over 2}\sum_{j\not=i}\frac{\lambda_j}{|\lambda_j-\lambda_{i}|}\right).\]
{\bf Remark:} Observe that $\Xi(n,t,\nu)$ could be bounded above
also by
\begin{equation}\label{alt}
\frac{4C^2_1}{n}\times \big(\lambda_{i^*}\tilde{\nu}_{i^*}^2  \big)\times t\times \Big({5\over
    2}\sum_{j\not=i^*}\frac{\lambda_j}{|\lambda_j-\lambda_{i^*}|}\Big)\end{equation}
Also note that when ${\lambda_i\over \spectralGap_{i}}\geq 1$, then
$$\max\left((1+\frac{\lambda_i}{\spectralGap_{i}} ),
2\frac{\lambda_i}{q^2\spectralGap_{i}}\right)=2\frac{\lambda_i}{q^2\spectralGap_{i}}$$
and so with the bound (\ref{alt}), $\Psi(q,t,i)$ reduces to
$$80C_1^2t{\lambda_i \over
    q^2\spectralGap_{i}}\Big(\sum_{j\not=i}\frac{\lambda_j}{|\lambda_j-\lambda_i|}\Big).$$
\\\\
Note that we can not apply Lemma \ref{rescaling lemma} directly,
since $\nu$ is a random vector. So to apply the lemma, we have to
get rid of randmomness of $\nu$ which we achieve by conditioning
$i^*$.  To simplify notations, let us define the vector $\nu(k)$ by
$\nu_j(k)=\sqrt{2}|\lambda_j-\lambda_k|^{-1/2},\ j\not= k$ and
\[\nu_k(k)=\max_{j\not=k} \frac{\sqrt{\lambda_j}}{\sqrt{\lambda_j-\lambda_{k}}\,\sqrt{\lambda_{k}}}.
\]
Thus, conditional on $i^*=k$ we have $\nu=\nu(k)$
and  $\|E(\nu)\|=\|E(\nu(k))\|$, as well as
$$\Xi(n,t,\nu)=\Xi\left(n,t,\nu(k)\right)\leq\Phi(n,t,k).$$
Thus, for fixed $p\in\N,  t\geq 1,\,  n\geq \max(p,t)$ and
$k\in\{1,\dots,p\}$ we have
\begin{align}
\Prob\left[\left\|E\left(\nu\right)\right\|^2\leq\Phi\left(n,t,i^*\right)\;\Bigr|\;i^*=k\right]&\geq
\Prob\left[\left\|E\left(\nu(k)\right)\right\|^2\leq\Xi\left(n,t,\nu(k)\right)\;\Bigr|\;i^*=k\right]\nonumber\\
 &=\frac{\Prob\left[\{i^*=k\}\setminus  \{\left\|E\left(\nu(k)\right)\right\|^2>\Xi\left(n,t,\nu(k)\right)\}\right]}{\Prob[i^*=k]}\nonumber\\
 &\geq \frac{\Prob\left[i^*=k\right] - \Prob\left[  \left\|E\left(\nu(k)\right)\right\|^2>\Xi\left(n,t,\nu(k)\right)\right]}{\Prob[i^*=k]}\nonumber\\
&\geq1-\dfrac{\min\left(\e^{-t},\;\Prob\left[i^*=k\right]\right)}{\Prob\left[i^*=k\right]}.\label{ahaseli}
\end{align}

We are ready to prove the main theorem:

\begin{proof}{\bf (Theorem \ref{main})}
For all $n,k,t$ and $q$ we have
\begin{equation}
n\geq\Psi(q,t,k)\Rightarrow \Phi(n,t,k) \leq \dfrac{1}{4}\label{lala1}
\end{equation}
and
\begin{equation}\label{lala2}
n\geq\Psi(q,t,k)\Rightarrow \Phi(n,t,k) \leq {q^2 \spectralGap_{k} \nu^2_k(k)\over 8}
\end{equation}
Therefore, we have
\begin{align}
\|E(\nu)\|^2\leq\Phi\left(n,t,i^*\right),\;&n\geq\Psi\left(q,t,i^*\right)\nonumber\\
&\stackrel{\eqref{um},\eqref{dois},\eqref{treis}},\eqref{lala1},\eqref{lala2}{\Rightarrow}
\|\Lambda_{i^*}\|\leq\dfrac{1}{2},\;
\sqrt{\dfrac{\lambda_{i^*}}{n}}\|\vec{V}_{i^*}\|\leq q\sqrt{\spectralGap_{i^*}\over 8}\nonumber\\
&\stackrel{\eqref{soso}}{\Rightarrow}\|\Delta\vec{\mu}_{i^*,i}\|\leq
\sqrt{\dfrac{2}{\spectralGap_{i^*}}}\times\dfrac{1}{1-\frac{1}{2}}\times q\sqrt{\spectralGap_{i^*}\over 8}=q\nonumber\\
&\stackrel{\eqref{halloechen}}{\Rightarrow}\sin\theta\leq q. \label{lala3}
\end{align}
Moreover, since $\Psi(q,t,i^*)$ is stochastic only through it's dependence on $i^*$, we have for $k$ which satisfy $n\geq \Psi(q,t,k)$ the equality
\begin{multline}
\Prob\left[\|E(\nu)\|^2\leq\Phi(n,t,i^*)\,\Bigr|\, i^*=k,\;n\geq\Psi(q,t,i^*)\right]
=\Prob\left[\|E(\nu)\|^2\leq\Phi(n,t,i^*)\,\Bigr|\, i^*=k\right].
\label{lala4}
\end{multline}
This yields (recall that by assumption $n\geq \Psi(q,t,i)$)
\begin{align*}
\Prob&\left[\sin\theta \leq q\;\Bigr|\;n\geq\Psi\left(q,t,i^*\right)\right]\\
&\stackrel{\eqref{lala3}}{\geq}
\Prob\left[\|E(\nu)\|^2\leq\Phi\left(n,t,i^*\right)\;\Bigr|\;n\geq\Psi\left(q,t,i^*\right)\right]\\
&\stackrel{\eqref{lala4}}{\geq}\sum_{k\,:\,n\geq \Psi(q,t,k)} \Prob\left[\|E(\nu)\|^2\leq\Phi\left(n,t,i^*\right)\;\Bigr|\;i^*=k\right]\times
\Prob\left[i^*=k\;\Bigr|\;n\geq\Psi\left(q,t,i^*\right)\right]\\
&\stackrel{\eqref{ahaseli}}{\geq}\sum_{k\,:\,n\geq \Psi(q,t,k)}\left(1-\dfrac{\min\left(\e^{-t},\;\Prob\left[i^*=k\right]\right)}{\Prob\left[i^*=k\right]}\right)
\times\Prob\left[i^*=k\;\Bigr|\;n\geq\Psi\left(q,t,i^*\right)\right]\\
&=1-\sum_{k\,:\,n\geq \Psi(q,t,k)} \dfrac{\min\left(\e^{-t},\;\Prob\left[i^*=k\right]\right)}{\Prob\left[n\geq\Psi(q,t,i^*)\right]}
\times\dfrac{\Prob\left[i^*=k,\;n\geq\Psi\left(q,t,i^*\right)\right]}{\Prob\left[i^*=k\right]}\\
&\geq1-\dfrac{\e^{-t}}{\Prob\left[n\geq\Psi(q,t,i^*)\right]}
\times\left(\sum_{k\,:\,n\geq \Psi(q,t,k)} \Prob\left[n\geq\Psi\left(q,t,i^*\right)\,\Bigr|\,i^*=k\right]\right).
\end{align*}
\end{proof}
\bibliographystyle{abbrv}

\noindent {\bf Acknowledgment:} R.\ Hauser would like to thank the
Alan Turing Institute London for generous support by The Alan Turing
Institute under the EPSRC grant EP/N510129/1. H.\ Matzinger would
like to thank Pembroke College Oxford and the Oxford Mathematical
Institute for hosting him as a Visiting Fellow and Academic Visitor,
respectively, during his sabbatical in 2017. J.\ Lember and R.\
Kangro are grateful for support by Estonian institutional research
funding IUT34-5.


\end{document}